\documentclass[11pt]{amsart}

\usepackage{amsmath}
\usepackage{amssymb}
\usepackage{amscd}

\usepackage{enumitem}
\usepackage{mathtools}

\usepackage{tikz-cd}

\newcommand{\nc}{\newcommand}

\DeclarePairedDelimiter\abs{\lvert}{\rvert}

\nc{\kk}{{\mathsf{k}}}

\nc{\PP}{{\mathbb{P}}}
\nc{\QQ}{{\mathbb{Q}}}
\nc{\RR}{{\mathbb{R}}}
\nc{\ZZ}{{\mathbb{Z}}}

\nc{\CA}{{\mathcal{A}}}
\nc{\CB}{{\mathcal{B}}}
\nc{\CC}{{\mathcal{C}}}
\nc{\D}{{\mathcal{D}}}
\nc{\CE}{{\mathcal{E}}}
\nc{\CF}{{\mathcal{F}}}
\nc{\CG}{{\mathcal{G}}}
\nc{\CH}{{\mathcal{H}}}
\nc{\CL}{{\mathcal{L}}}
\nc{\CM}{{\mathcal{M}}}
\nc{\CN}{{\mathcal{N}}}
\nc{\CO}{{\mathcal{O}}}
\nc{\CQ}{{\mathcal{Q}}}
\nc{\CR}{{\mathcal{R}}}
\nc{\CS}{{\mathcal{S}}}
\nc{\CT}{{\mathcal{T}}}
\nc{\CU}{{\mathcal{U}}}
\nc{\CV}{{\mathcal{V}}}
\nc{\CW}{{\mathcal{W}}}
\nc{\CX}{{\mathcal{X}}}
\nc{\CY}{{\mathcal{Y}}}

\nc{\eps}{\varepsilon}

\nc{\lotimes}{\mathbin{\mathop{\otimes}\limits^{\mathbb{L}}}}
\nc{\CEnd}{\mathop{\mathcal{E}\mathit{nd}}\nolimits}
\nc{\CExt}{\mathop{\mathcal{E}\mathit{xt}}\nolimits}
\nc{\CHom}{\mathop{\mathcal{H}\mathit{om}}\nolimits}
\nc{\RGamma}{\mathop{{\mathsf{R}}\Gamma}\nolimits}
\nc{\RHom}{\mathop{\mathsf{RHom}}\nolimits}
\nc{\RCHom}{\mathop{\mathsf{R}\mathcal{H}\mathit{om}}\nolimits}
\nc{\RG}{\mathop{\mathsf{R\Gamma}}\nolimits}
\nc{\Hom}{\mathop{\mathsf{Hom}}\nolimits}
\nc{\Ext}{\mathop{\mathsf{Ext}}\nolimits}
\nc{\End}{\mathop{\mathsf{End}}\nolimits}
\nc{\Tor}{\mathop{\mathsf{Tor}}\nolimits}
\nc{\Tordim}{\mathop{\mathsf{Tor}\text{\rm-}\mathsf{dim}}\nolimits}
\nc{\Hilb}{\mathop{\mathsf{Hilb}}\nolimits}
\nc{\Spec}{\mathop{\mathsf{Spec}}\nolimits}
\nc{\Pic}{\mathop{\mathsf{Pic}}\nolimits}

\nc{\Tr}{\mathop{\mathsf{Tr}}\nolimits}
\nc{\Cone}{\mathop{\mathsf{Cone}}\nolimits}
\nc{\Ker}{\mathop{\mathsf{Ker}}\nolimits}
\nc{\Coker}{\mathop{\mathsf{Coker}}\nolimits}
\nc{\codim}{\mathop{\mathsf{codim}}\nolimits}
\nc{\sing}{{\mathsf{sing}}}
\nc{\supp}{\mathop{\mathsf{supp}}}
\nc{\vol}{\mathop{\mathsf{vol}}\nolimits}
\nc{\ch}{\mathop{\mathsf{ch}}\nolimits}
\nc{\perf}{{\mathsf{perf}}}
\nc{\rank}{\mathop{\mathsf{rank}}}
\nc{\rk}{\mathop{\mathsf{rk}}}

\nc{\Pf}{{\mathsf{Pf}}}
\nc{\Gr}{{\mathsf{Gr}}}
\nc{\OGr}{{\mathsf{OGr}}}
\nc{\LGr}{{\mathsf{LGr}}}
\nc{\IGr}{{\mathsf{IGr}}}
\nc{\IFl}{{\mathsf{IFl}}}
\nc{\OF}{{\mathsf{OF}}}
\nc{\Fl}{{\mathsf{Fl}}}
\nc{\Bl}{{\mathsf{Bl}}}
\nc{\GL}{{\mathsf{GL}}}
\nc{\PGL}{{\mathsf{PGL}}}
\nc{\SL}{{\mathsf{SL}}}
\nc{\SP}{{\mathsf{Sp}}}
\nc{\Spin}{{\mathsf{Spin}}}
\nc{\Tot}{{\mathsf{Tot}}}
\nc{\ev}{{\mathsf{ev}}}
\nc{\od}{{\mathsf{odd}}}
\nc{\coev}{{\mathsf{coev}}}
\nc{\id}{{\mathsf{id}}}
\nc{\opp}{{\mathsf{opp}}}
\nc{\tdim}{\mathop{\Tor\dim}}
\nc{\ad}{{\mathop{\mathsf ad}}}
\nc{\sg}{{\mathop{\mathsf sg}}}
\nc{\hf}{{\mathop{\mathsf hf}}}
\nc{\gr}{{\mathop{\mathsf gr}}}
\nc{\qgr}{{\mathop{\mathsf qgr}}}
\nc{\Coh}{{\mathop{\mathsf Coh}}}
\nc{\fsl}{{\mathfrak{sl}}}
\nc{\fso}{{\mathfrak{so}}}
\nc{\fgl}{{\mathfrak{gl}}}

\nc{\Rep}{{\mathsf{Rep}}}

\nc{\Xt}{{\tilde{X}}}
\nc{\Xb}{{\bar{X}}}
\nc{\CUt}{{\tilde{\CU}}}
\nc{\CAt}{{\tilde{\CA}}}
\nc{\CUb}{{\bar{\CU}}}
\nc{\Hb}{{\bar{H}}}
\nc{\pt}{\tilde{p}}
\nc{\Vb}{\bar{V}}
\nc{\io}{{\iota}}
\nc{\iot}{{\tilde{\iota}}}
\nc{\lort}[1]{\vphantom{{#1}}^\perp{#1}}

\theoremstyle{plain}

\newtheorem*{theorem*}{Theorem}
\newtheorem{theorem}{Theorem}[section]
\newtheorem{conjecture}[theorem]{Conjecture}
\newtheorem{lemma}[theorem]{Lemma}
\newtheorem{proposition}[theorem]{Proposition}
\newtheorem{corollary}[theorem]{Corollary}

\theoremstyle{definition}

\newtheorem{definition}[theorem]{Definition}

\theoremstyle{remark}

\newtheorem{remark}[theorem]{Remark}

\title{On the bounded derived category of $\IGr(3,7)$}
\author{Anton Fonarev}
\address{\sloppy
\parbox{0.95\textwidth}{
   Algebraic Geometry Section, Steklov Mathematical Institute of Russian Academy of Sciences,
8 Gubkin str., Moscow 119991 Russia
\hfill\\[5pt]
National Research University Higher School of Economics, Russian Federation,
Laboratory of Mirror Symmetry, NRU HSE, 6 Usacheva str., Moscow, Russia, 119048
\hfill
}\bigskip}
\email{avfonarev@mi.ras.ru}
\date{}
\thanks{The author is partially supported by Laboratory of Mirror Symmetry NRU HSE, RF Government grant, ag. No 14.641.31.0001.
  The author is a ``Young Russian Mathematics'' award winner and a Simons-IUM fellow
  and would like to thank its sponsors and jury.}

\begin{document}

\begin{abstract}
   We construct a minimal Lefschetz decomposition of the bounded derived category
   of coherent sheaves on the isotropic Grassmannian $\IGr(3,7)$. Moreover, we show
   that $\IGr(3, 7)$ admits a full exceptional collection consisting of equivariant
   vector bundles.
\end{abstract}

\maketitle

\section{Introduction}

One of the most important invariants of a smooth projective variety $X$ is the
bounded derived category $D^b(X)$ of coherent sheaves on it. As it often
happens, the bounded derived category (from now on we will drop the word bounded)
is rather easy to define, but quite difficult to describe explicitly.
A rather fruitful approach to the latter problem is to split the derived
category into smaller pieces, which is precisely the notion of a semiorthogonal
decomposition. Then one can study the components of a given decomposition
on their own and, finally, how they are glued together.

In the best case scenario one can decompose the derived category in such a way
that the pieces are as simple as they can get: equivalent to the derived category
of a point, which is in turn equivalent to the category of graded finite-dimensional
vector spaces over the base field. Such a decomposition corresponds to what is called a full
exceptional collection. The study of exceptional
collections goes back to seminal works of A.~Beilinson (\cite{beilinson1978coherent})
and M.~Kapranov (\cite{kapranov1988derived}), where it was shown that projective
spaces and, more generally, Grassmannians admit full exceptional collections
consisting of vector bundles.

Having a full exceptional collection is a very strong condition on the variety.
First of all, it is very easy to give an example of a variety whose derived
category does not admit any nontrivial semiorthogonal decompositions at all:
any smooth projective curve of positive genus would suffice (see~\cite{okawa2011semi}).
Next, the Grothendieck group of the variety, $K_0(X)$, must necessarily be free.
Also, it was recently shown that the integral Chow motive of a smooth projective variety
of dimension at most~3 which admits a full exceptional collection is
of Lefschetz type (see~\cite{Gorchinskiy20171827}). Finally, a conjecture by
D.~Orlov predicts that having a full exceptional collection implies rationality.
It is worth mentioning that even though varieties with a full exceptional collection
are rare, it was recently proved that any triangulated category with a full exceptional
collection is of geometric nature, that is, can be embedded in the derived category
of a smooth projective variety (see~\cite{orlov2015geometric,orlov2016smooth,orlov2016gluing}).

Since the work of Beilinson and Kapranov,
it has been conjectured that the bounded derived category
of a rational homogeneous variety admits a full exceptional collection.
Unfortunately, even for general isotropic Grassmannians the best result up to date
is the work of A.~Kuznetsov and A.~Polishchuk
where exceptional collections of maximal possible length are constructed in the derived
categories of the latter (see~\cite{kuznetsov2016exceptional}).
While the general conjecture remains open, there seems to be a slightly overlooked
class of varieties that could shed some light upon the general case. These are
the so-called odd isotropic Grassmannians. Given a vector space $V$ of odd dimension $(2n+1)$
over a field $\kk$ together with a skew-symmetric form $\omega\in\Lambda^2V^*$
of maximal rank, one can look at the variety $\IGr_\omega(k, V)$ parametrizing
$k$-dimensional subspaces in $V$ isotropic with respect to $\omega$. The classical
geometry of isotropic Grassmannians and, more generally, flag varieties was first studied
by Mihai in~\cite{mihai2007odd}. It seems very natural to conjecture that
odd isotropic Grassmannians admit full exceptional collections. While isotropic
varieties of lines are nothing but projective spaces, the case of isotropic planes
can be treated rather easily as $\IGr(2, 2n+1)$ is a hyperplane section of $\Gr(2, 2n+1)$
(see~\cite{kuznetsov2008exceptional, pech2013quantum}). Also, odd isotropic Grassmannians are examples of
horospherical varieties with Picard rank~1 (see~\cite{pasquier2009some}).
A conjecturally full exceptional collection in the derived
category of the horospherical variety related to the group $G_2$ was constructed recently 
in~\cite{gonzales2018geometry}.

In the present paper we construct a full exceptional collection in $D^b(\IGr(3,7))$,
which is the fist uncovered example. This variety is tightly related
to one of the K\"uchle
fourfolds, Fano varieties constructed as zero loci of sections of equivariant vector
bundles on Grassmannians (see~\cite{kuchle1995fano,kuznetsov2015kuchle,kuznetsov2016kuchle}). 
The exceptional collection we construct has a nice block structure (is rectangular Lefschetz),
and we hope that it will
help determine the structure of the derived category of the corresponding K\"uchle variety.

Let $\CU$ and $\CQ$ denote the universal sub and quotient bundles on $\IGr(3,7)$ respectively.

\begin{theorem*}
  The bounded derived category of coherent sheaves on $\IGr(3, 7)$ admits a full rectangular
  Lefschetz exceptional collection consisting of the vector bundles
  
  \begin{equation*}
    D^b(\IGr(3, 7)) =
      \begin{pmatrix*}[r]
        \Lambda^2\CQ & \Lambda^2\CQ(1) & \Lambda^2\CQ(2) & \Lambda^2\CQ(3) & \Lambda^2\CQ(4) \\
        \CU^* & \CU^*(1) & \CU^*(2) & \CU^*(3) & \CU^*(4) \\
        \CO & \CO(1) & \CO(2) & \CO(3) & \CO(4) \\
        \CU & \CU(1) & \CU(2) & \CU(3) & \CU(4) \\
      \end{pmatrix*}.
  \end{equation*}
\end{theorem*}

We work over an algebraically closed filed $\kk$ of characteristic~0.
The paper is organized as follows. In the second section we collect preliminaries
form the theory of derived categories and equivariant vector bundles on rational
homogeneous varieties. In the third section we introduce the geometric setting and
construct some short exact sequences of coherent sheaves that are essential to the proof
of the main theorem. Finally, in the fourth section we prove some vanishing statements
for equivariant vector bundles on isotropic Grassmannians and give a proof of the main
theorem.

\subsection*{Acknowledgments}
I would like to thank Alexander Kuznetsov and Dmitry Orlov for several useful conversations.
I would also like to thank University of Geneva and Andras Szenes for hospitality and
excellent working conditions under which the present paper was partially written.

\section{Preliminaries}

\subsection{Schur functors}
Given a positive integer $n$, we denote by $P^+_n$ the set of
weakly decreasing sequences of $n$ integer numbers
\begin{equation*}
  P^+_n = \left\{ \lambda\in \ZZ^n \mid \lambda_1\geq \lambda_2\geq \ldots \geq \lambda_n \right\},
\end{equation*}
and identify it with the set of dominant weights for $\GL_n$.
There is a natural partial inclusion order on $P^+_n$ given by
\begin{equation*}
  \mu\subseteq\lambda\quad \Leftrightarrow\quad \mu_i\leq\lambda_i\quad \text{for all}\quad i=1,\ldots,n.
\end{equation*}
We denote by $\Sigma^\lambda$ the Schur functor corresponding to $\lambda\in P^+_n$.
Our convention is that $\Sigma^{(k,0,\ldots,0)}=S^k$ is the $k$-th symmetric power functor.
Given a pair of elements $\mu,\lambda\in P^+_n$ such that $\mu\subseteq\lambda$, we put
$|\lambda/\mu| = \Sigma_i (\lambda_i-\mu_i)$.

Let $\CU$ be a vector bundle of rank $n$ on a smooth algebraic variety. Then
$\Sigma^\lambda\CU\simeq\Sigma^{-\lambda}\CU^*$, where
$-\lambda=(-\lambda_n, -\lambda_{n-1},\ldots,-\lambda_1)$.
Given $\lambda,\mu\in P^+_n$, the Littlewood--Richardson rule provides a recipe to
decompose the tensor product $\Sigma^\lambda\CU\otimes\Sigma^\mu\CU$
into a direct sum of bundles of the form $\Sigma^\alpha\CU$.
In the present paper we will need the simplest case of the Littlewood--Richardson rule,
namely, Pieri's formulas.

\begin{lemma}[Pieri's Formulas]
  \label{lm:pieri}
  Let $\lambda\in P^+_n$, let $\CU$ be a vector bundle on a smooth algebraic variety,
  and let $k$ be a positive integer. There are isomorphisms
  \begin{equation*}
    \Sigma^\lambda\CU\otimes S^k\CU \simeq \bigoplus_{\mu\in HS_\lambda^k}\Sigma^\mu\CU
    \quad\text{and}\quad
    \Sigma^\lambda\CU\otimes \Lambda^k\CU \simeq \bigoplus_{\mu\in VS_\lambda^k}\Sigma^\mu\CU,
  \end{equation*}
  where
  \begin{align*}
    HS_\lambda^k &= \left\{ \mu\in P^+_n \mid \mu\supseteq\lambda,\  |\mu/\lambda|=k,\ \text{and}\ \mu_1\geq\lambda_1\geq\mu_2\geq\lambda_2\geq\ldots\geq\mu_n\geq\lambda_n \right\}, \\
    VS_\lambda^k &= \{ \mu\in P^+_n \mid \mu\supseteq\lambda,\ |\mu/\lambda|=k,\ \text{and}\ \lambda_i+1\geq\mu_i\geq\lambda_i\quad \text{for all}\quad i=1,\ldots,n \}.
  \end{align*}
\end{lemma}

Let $Y_n\subset P^+_n$ denote the set of those $\lambda\in P^+_n$ for which
$\lambda_n\geq 0$. This subset is naturally identified with the set of Young
diagrams with at most $n$ rows.
For a given $\lambda\in Y_n$ let
\begin{equation*}
  d(\lambda)=\max\{i \mid \lambda_i\geq i\}
\end{equation*}
denote the length of its diagonal, 
let $\lambda^T\in P^+_{\lambda_1}$ denote the  transposed Young diagram,
and let $|\lambda|=\Sigma \lambda_i$ denote the number of boxes in $\lambda$. 
We say that $\lambda$ is symmetric if $\lambda=\lambda^T$ and that
$\lambda$ is \emph{almost symmetric} if the diagram
\begin{equation*}
  (\underbrace{\lambda^T_1-1, \lambda^T_2-1, \ldots, \lambda^T_{d(\lambda)}-1}_{d(\lambda)\text{ terms}}, \lambda^T_{d(\lambda)+1},\ldots,\lambda^T_{\lambda_1})
\end{equation*}
is symmetric.

As before, let $\CU$ be a vector bundle of rank $n$ on a smooth projective variety.
The importance of the class of almost symmetric diagrams is illustrated by the
following lemma.

\begin{lemma}\label{lm:ll2}
  There is an isomorphism of vector bundles
  $\Lambda^k\Lambda^2\CU\simeq \bigoplus_{\lambda\in L_k}\Sigma^\lambda\CU$,
  where $L_k\subseteq Y_n$ denotes the set of almost symmetric diagrams with $2k$ boxes.
\end{lemma}

We refer the reader to~\cite{weyman} for further details.

\subsection{Borel--Bott--Weil theorem}

The celebrated Borel--Bott--Weil theorem fully computes cohomology
groups of irreducible equivariant vector bundles on rational homogeneous
varieties. In what follows we present stripped down versions of it for
classical and isotropic Grassmannians.
For all the details we refer the reader to~\cite{weyman}.

Let $\Gr(k, V)$ denote the Grassmannian of $k$-dimensional subspaces
in a fixed $n$-dimensional vector space $V$. Let $\CU$ and $\CQ$ denote
the universal sub and quotient bundles respectively. For instance, one
has the following short exact sequence of vector bundles on $\Gr(k, V)$:
\begin{equation*}
  0\to \CU\to V\otimes\CO\to \CQ\to 0.
\end{equation*}
The Grassmannian comes with a natural action of the linear algebraic
group $\GL(V)$.
It is well known that every irreducible equivariant vector bundle on
$\Gr(k, V)$ is of the form $\Sigma^\lambda\CU^*\otimes \Sigma^\mu\CQ^*$
for some $\lambda\in P^+_k$ and $\mu\in P^+_{n-k}$.
Given a permutation $\sigma\in\mathfrak{S}_n$ of the set $\{1,2,\ldots,n\}$,
let $\ell(\sigma)$ denote the number of pairs of indices
$1\leq p < q\leq n$ such that $\sigma(p) > \sigma(q)$.

\begin{theorem}
  \label{thm:bbw}
  Consider the sequence
  \begin{equation*}
    \alpha=(n+\lambda_1,\, (n-1)+\lambda_2,\,\ldots,\, (n-k+1)+\lambda_k,\, n-k+\mu_1,\, \ldots,\, 2+\mu_{n-k-1},\, 1+\mu_n).
  \end{equation*}
  If at least two of the elements in $\alpha$ are equal, then
  \begin{equation*}
    H^\bullet(\Gr(k, V), \Sigma^\lambda\CU^*\otimes\Sigma^\mu\CQ^*) = 0.
  \end{equation*}
  If all the elements in $\alpha$ are distinct, let $\sigma\in\mathfrak{S}_n$
  denote the unique permutation such that
  \begin{equation*}
    \alpha_{\sigma(1)}>\alpha_{\sigma(2)}>\ldots>\alpha_{\sigma(n)},
  \end{equation*}
  and let
  \begin{equation*}
    \gamma = (\alpha_{\sigma(1)}-n,\, \alpha_{\sigma(2)}-(n-1),\,\ldots,\, \alpha_{\sigma(n)}-1).
  \end{equation*}
  Then
  \begin{equation*}
    H^i(\Gr(k, V), \Sigma^\lambda\CU^*\otimes\Sigma^\mu\CQ^*) = \begin{cases}
      \Sigma^\gamma V^*, & \text{if } i = \ell(\sigma), \\
      0 & \text{otherwise.}
    \end{cases}
  \end{equation*}
\end{theorem}

The following lemma first appeared in Kapranov's work on the derived
categories of classical Grassmannians and is a simple corollary of the
previous theorem.

\begin{lemma}[{\cite[Lemma~3.2]{kapranov1988derived}}]\label{lm:kap-dual}
  Let $X=\Gr(k, n)$, and let $\lambda\in Y_k$ and $\mu\in Y_{n-k}$ be such that
  \begin{equation*}
    n-k\geq \lambda_1\geq\lambda_2\geq\ldots\geq\lambda_k\geq 0 \quad \text{and} \quad
    k\geq \mu_1 \geq \mu_2 \geq \ldots \geq \mu_{n-k} \geq 0.
  \end{equation*}
  Then
  \begin{equation*}
    H^\bullet(X, \Sigma^\lambda\CU\otimes\Sigma^\mu \CQ^*) =
    \begin{cases}
      \kk\left[-|\lambda|\right], & \text{if } \lambda = \mu^T, \\
      0, & \text{otherwise}.
    \end{cases}
  \end{equation*}
\end{lemma}

\begin{remark}\label{rm:kap-dual}
  Lemma~\ref{lm:kap-dual} follows directly from the Borel--Bott--Weil theorem
  and a simple combinatorial statement. In its weak form the statement says that
  under the assumptions of the lemma the elements of the sequence
  \begin{equation}\label{eq:kap-dual}
    \alpha=\left(n - \lambda_k, (n-1)-\lambda_{k-1},\ldots, (n-k+1)-\lambda_1,
      (n-k)+\mu_1, (n-k-1)+\mu_2, \ldots, 1+\mu_{n-k}\right)
  \end{equation}
  are distinct if an only if $\lambda=\mu^T$. Remark that the statement is
  translation invariant, that is, it remains true if one transforms
  $\alpha$ by adding a fixed constant to every term.
\end{remark}

Using the previous remark we can formulate the following simple vanishing criterion.

\begin{lemma}\label{lm:kap-gen}
  Let $\lambda\in P^+_n$ and $\mu\in P^+_{n-k}$ be such that for some integers
  $1\leq p\leq k$ and $1\leq q\leq n-k$ one has
  \begin{equation*}
    q\geq \lambda_1\geq \ldots \geq \lambda_p\geq 0
    \quad \text{and} \quad
    p\geq \mu_1\geq\ldots\geq \mu_q\geq 0.
  \end{equation*}
  If $(\lambda_1,\ldots,\lambda_p)\neq (\mu_1,\ldots,\mu_q)^T$, then
  \begin{equation*}
    H^\bullet(X, \Sigma^\lambda\CU\otimes\Sigma^\mu \CQ^*) = 0.
  \end{equation*}
\end{lemma}
\begin{proof}
  According to the Borel--Bott--Weil theorem, it is enough to show that the elements
  of the sequence $\alpha$ are not distinct. Now, it follows from Remark~\ref{rm:kap-dual}
  that the elements of the subsequence
  \begin{equation*}
    \left((n-k+p)-\lambda_p, \ldots, (n-k+1)-\lambda_1, (n-k)+\mu_1, \ldots, (n-k+1-q) + \mu_q\right)
  \end{equation*}
  are not distinct.
\end{proof}

Now, let $V$ be a $2n$-dimensional vector space with a fixed symplectic form
and let $\IGr(k, V)$ denote the Grassmannian of isotropic $k$-dimensional subspaces.
Again, we denote by $\CU$ the tautological rank $k$ bundle on $\IGr(k, V)$,
which is the pullback of the tautological bundle on $\Gr(k,V)$ under the natural
embedding $\IGr(k, V)\hookrightarrow \Gr(k, V)$.
The following is a simple corollary from the fully fledged Borel--Bott--Weil
theorem for isotropic Grassmannians.

\begin{proposition}\label{thm:bbwc}
  Given an element $\lambda\in P^+_k$, consider the sequence
  \begin{equation*}
    \alpha = (n+\lambda_1,\ldots, n-k+1+\lambda_k, n-k, \ldots, 1).
  \end{equation*}
  If at least two elements in the sequence $\alpha$ have equal
  absolute values, then
  \begin{equation*}
    H^\bullet(\IGr(k, V), \Sigma^\lambda\CU^*) = 0.
  \end{equation*}
\end{proposition}

\subsection{Semiorthogonal decompositions}

Let $\CT$ be a triangulated category.
\begin{definition}
  A sequence of full triangulated subcategories
  $\CA_0,\CA_1,\ldots,\CA_n\in \CT$ is called \emph{semiorthogonal} if for all $0\leq i<j\leq n$
  and all $E\in\CA_i$, $F\in\CA_j$ one has $\Hom_\CT(F, E)=0$.
  Let $\langle\CA_0,\CA_1,\ldots,\CA_n\rangle$ denote the smallest full triangulated
  subcategory in $\CT$ containing all $\CA_i$. If $\CT=\langle\CA_0,\CA_1,\ldots,\CA_n\rangle$,
  say that the subcategories $\CA_i$ form a \emph{semiorthogonal decomposition} of $\CT$.
\end{definition}

In nice situations one can construct a semiorthogonal decomposition from a given full triangulated
subcategory alone.
\begin{definition}
  A full triangulated subcategory $\CA\subseteq\CT$ is called \emph{admissible} if the embedding
  functor $\iota:\CA\to\CT$ has both left and right adjoints.
\end{definition}
All the subcategories considered in the present paper are admissible, see, for example,
\cite{bondal1990representable}.
Starting with an admissible subcategory $\CA\subseteq\CT$, one can produce two semiorthogonal decompositions.
Put
\begin{align*}
  \CA^\perp &= \langle X \in \CT \mid \Hom_\CT(Y,X)=0\ \text{for all}\ Y\in\CA\rangle, \\
  \lort{\CA} &= \langle X \in \CT \mid \Hom_\CT(X,Y)=0\ \text{for all}\ Y\in\CA\rangle.
\end{align*}
Then $\CT$ admits semiorthogonal decompositions
\begin{equation*}
  \CT=\langle \CA^\perp, \CA\rangle\quad \text{and}\quad \CT=\langle \CA, \lort{\CA}\rangle.
\end{equation*}
Moreover, for $\CT$ smooth and proper (which is always the case in the present paper)
one can show that both $\CA^\perp$ and $\lort{\CA}$ are admissible as well.
In what follows we denote by $L_\CA$ and $R_\CA$ the \emph{mutation} functors,
which are defined by the following properties:
for any object $X\in\CT$ there exist unique up to isomorphism functorial exact triangles
\begin{equation*}
  R_\CA X\to X\to Y'\to R_\CA X[1], \qquad
  Y'' \to X \to L_\CA X \to Y''[1],
\end{equation*}
with $Y',Y''\in\CA$, $R_\CA X\in\lort{\CA}$, and $L_\CA X\in\CA^\perp$. Both functors
vanish on $\CA$ and define mutually inverse equivalences between $\CA^\perp$ and $\lort{\CA}$.

Given an algebraic variety $X$, we denote by $D^b(X)$ the bounded derived category
of coherent sheaves on~$X$. Let $\CO(1)$ be a line bundle on $X$.
Recall that tensoring with a line bundle is an autoequivalence of $D^b(X)$.
Given a full triangulated subcategory $\CA\subseteq D^b(X)$, denote by $\CA(i)$ the full triangulated
subcategory, which is the image of $\CA$ under the autoequivalence $-\otimes\CO(i)$, where $\CO(i)$
is the $i$-th tensor power of $\CO(1)$.

Let us recall a couple of seminal results, both due to Orlov, along with two simple lemmas,
both of which must be known to specialists.

\begin{theorem}[Orlov's projective bundle formula, \cite{orlov1992projective}]
  Let $X$ be a smooth projective variety, and let $\CE$ be a vector bundle on $X$ of rank $n$.
  Let $p:\PP_X(\CE)\to X$ be the projectivization of $\CE$, and let $\CO(1)$ denote the Grothendieck
  line bundle. Then there is a semiorthogonal decomposition
  \begin{equation}
    \label{eq:orlov_proj}
    D^b(\PP_X(\CE)) = \langle p^*D^b(X)(-n+1), \ldots, p^*D^b(X)(-1), p^*D^b(X)\rangle.
  \end{equation}
\end{theorem}

The following lemma allows to mutate objects from the rightmost component of~\eqref{eq:orlov_proj}
all the way to the left.

\begin{lemma}\label{lm:mutproj}
  Let $X$ be a smooth projective variety, and let $\CE$ be a vector bundle on $X$ of rank $n$.
  Let $p:\PP_X(\CE)\to X$ be the projectivization of $\CE$, and let $\CO(1)$ denote the
  Grothendieck line bundle.
  Let $\CA=\left(p^*D^b(X)\right)^\perp$ and let
  $\CF\in D^b(X)$. Then, up to a shift, $L_\CA(p^*\CF)\simeq p^*\left(\CF\otimes\det\CE^*\right)(-n)$.
\end{lemma}
\begin{proof}
  Recall that $\CA=\langle p^*D^b(X)(-n+1), \ldots, p^*D^b(X)(-1)\rangle$. Applying the autoequivalence
  $-\otimes \CO(-1)$ to the decomposition~\eqref{eq:orlov_proj}, we see that $\CA^\perp=p^*D^b(X)(-n)$.
  Finally, consider the tautological embedding $\CO(-1)\to p^*\CE$. The corresponding section of
  $p^*\CE(1)$ is nowhere vanishing, thus, there is a Koszul complex
  \begin{equation}\label{eq:projkosz}
    0\to \Lambda^n \left(p^*\CE^*\right) (-n)\to \Lambda^{n-1}\left(p^*\CE^*\right) (-n+1)\to \cdots\to p^*\CE^*(-1)\to \CO\to 0.
  \end{equation}
  The complex~\eqref{eq:projkosz} induces an exact triangle in $D^b(\PP_X(\CE))$ of the form
  \begin{equation}\label{eq:ptojtri}
    p^*\det\CE^*(-n)\to X\to \CO_E[2-n]\xrightarrow{+1}\ ,
  \end{equation}
  where the object
  \begin{equation*}
    X=\left\{\cdots\to 0\to \Lambda^{n-1}\left(p^*\CE^*\right) (-n+1)\to \cdots\to p^*\CE^*(-1)\to 0
      \to \cdots\right\}
  \end{equation*}
  obviously belongs to the subcategory $\CA$.
  After the complex with $p^*\CF$, we get an exact triangle of the form
  \begin{equation*}
    p^*\left(\CF\otimes\det\CE^*\right)(-n) \to p^*\CF\otimes X 
    \to p^*\CF[2-n]\xrightarrow{+1}.
  \end{equation*}
  The subcategory $\CA$ is stable under tensor product with objects pulled
  back from $X$, thus $p^*\CF\otimes X\in \CA$ and the desired result follows.
\end{proof}

\begin{theorem}[Orlov's blow up formula, \cite{orlov1992projective}]
  Consider the commutative diagram
  \begin{equation*}
    \begin{tikzcd}
      E \rar{\iot} \dar{\pt} & \Xt \dar{p} \\
      Z \rar{\iota} & X
    \end{tikzcd}
  \end{equation*}
  where $X$ is a smooth projective variety, $Z$ is a smooth projective subvariety of $X$ of codimension $c$,
  $p:\Xt=\Bl_ZX\to X$ is the blow-up of $X$ in $Z$, and $E$ is the exceptional divisor.
  The derived category $D^b(X)$ admits the semiorthogonal decomposition 
  \begin{equation*}
    D^b(X)=\langle \iot_*\pt^*D^b(Z)((c-1)E),\ldots,\iot_*\pt^*D^b(Z)(E), p^*D^b(X)\rangle.
  \end{equation*}
\end{theorem}

\begin{lemma}\label{lm:blmut}
   Consider the commutative diagram
   \begin{equation*}
      \begin{tikzcd}
         E \rar{\iot} \dar{\pt} & \Xt \dar{p} \\
         Z \rar{\iota} & X
      \end{tikzcd}
   \end{equation*}
   where $X$ is a smooth projective variety, $Z$ is a smooth projective subvariety of $X$ of codimension $c$,
   $p:\Xt=\Bl_ZX\to X$ is the blow-up of $X$ in $Z$, and $E$ is the exceptional divisor,
   and let $\CN$ denote the normal bundle to $Z$ in $X$ as well as its pullback on $E$.
   Let $\CA$ be a full triangulated subcategory in $D^b(\Xt)$ such that
   \begin{equation*}
    \langle \iot_*\pt^*D^b(Z)((c-1)E),\ldots,\iot_*\pt^*D^b(Z)(E) \rangle \subseteq \CA,
   \end{equation*}
   and let $\CF$ be an object in $D^b(Z)$.
   Then $\iot_*\pt^*\CF$ belongs to $\CA$ if and only if
   $\iot_*\pt^*\left(\CF\otimes\det\CN^*\right)(cE)$ does.
\end{lemma}
\begin{proof}
  The proof is very similar to the proof of Lemma~\ref{lm:mutproj}.
  Recall that $E=\PP_Z(\CN_{X/Z})$, and that $\iot^*\CO(E)\simeq \CO(-1)$.
  The Koszul complex~\eqref{eq:projkosz} induces an exact triangle in $D^b(E)$ of the form
  \begin{equation}\label{eq:blow_tr1}
    \det\CN^*(-c)\to X\to \CO_E[2-c]\xrightarrow{+1}\ ,
  \end{equation}
  where $X$ is the complex Koszul complex stupidly truncated from both sides,
  \begin{equation*}
    X=\left\{\cdots\to 0\to \Lambda^{c-1}\CN^*(-c+1)\to \cdots\to \Lambda^2\CN^*(-2)\to \CN^*(-1)\to 0
      \to \cdots\right\}.
  \end{equation*}
  By projection formula
  $\iot_*\!\left(\pt^*\CF\otimes\Lambda^i\CN^*(-i)\right)\simeq\iot_*\pt^*\!\left(\CF\otimes\Lambda^i\CN^*\right)(iE)$.
  Applying the functor $\iot_*(\pt^*\CF\otimes-)$ to the triangle~\eqref{eq:blow_tr1}, we get
  and exact triangle in $D^b(\Xt)$ of the form
  \begin{equation*}
    \iot_*\pt^*(\CF\otimes\det\CN^*)(cE)\to \iot_*(\pt^*\CF\otimes X)\to
    \iot_*\pt^*\CF[2-c]\xrightarrow{+1}.
  \end{equation*}
  It remains to observe that $\iot_*(\pt^*\CF\otimes X)\in\langle
  \iot_*\pt^*D^b(Z)((c-1)E),\ldots,\iot_*\pt^*D^b(Z)(E) \rangle\subseteq\CA$.
\end{proof}

\subsection{Exceptional collections}

Let us now assume that $\CT$ is a $\kk$-linear triangulated category.
Then one can look for full triangulated subcategories of
simplest possible form, namely, equivalent to the derived category of a point.
These are in one to one
correspondence with the so called exceptional objects considered up to a shift.
Recall that for a general triangulated category
extension groups are defined as $\Ext^i(E, F)=\Hom(E, F[i])$.

\begin{definition}
  An object $E\in\CT$ is called \emph{exceptional} if
  \begin{equation*}
    \Ext^i(E, E) =
    \begin{cases}
      \kk, & \text{when } i=0, \\
      0,   & \text{otherwise}.
    \end{cases}
  \end{equation*}
\end{definition}
If $E\in\CT$ is exceptional, then the subcategory $\langle E\rangle\subseteq\CT$
is admissible and equivalent to the derived category of a point.
Exceptional objects are obviously stable under autoequivalences, in particular, the shift functor.

Similar to the notion of a semiorthogonal decomposition, there is a notion of an exceptional collection.

\begin{definition}
  A collection of objects $E_0,E_1,\ldots,E_n\in\CT$ is called \emph{exceptional} if for all $0\leq i \leq n$
  the objects $E_i$ are exceptional, and for all $0\leq i<j\leq n$ one has $\Ext^\bullet(E_j,E_i)=0$.
  An exceptional collection is called \emph{full} if $\CT$ is the smallest
  triangulated subcategory, containing all $E_i$, that is, $\CT=\langle E_0,E_1,\ldots,E_n\rangle$.
\end{definition}

Given an exceptional object $E$ and an arbitrary object $F$, one can check that $L_EF$
and $R_EF$ (we omit angular
brackets in subscripts when dealing with mutations through subcategories generated by a single exceptional
object) fit into exact triangles
\begin{equation*}
  L_EF[-1]\to \Hom^\bullet(E, F)\otimes E \to F \to L_EF,
  \qquad R_EF\to F\to \Hom^\bullet(F, E)^*\otimes E\to R_EF[1].
\end{equation*}
Given an admissible subcategory $\CA\subseteq\CT$ and an exceptional object
$E\in \CA^\perp$, the object $E'=R_\CA E$ is determined by the properties
\begin{equation*}
  E'\in \langle E, \CA\rangle,\quad E'\in \lort{\CA},\quad\text{and}\quad \Ext^\bullet(E',E)\simeq\kk
\end{equation*}
and is exceptional. A similar criterion can be written for an exceptional $E\in\lort{\CA}$
and $E'=L_\CA E$.

\subsection{Lefschetz decompositions}
The main reference for this section is~\cite{kuznetsov2008lefschetz}.
We return to the setting of a smooth projective algebraic variety $X$ equipped with a line bundle
$\CO(1)$.

\begin{definition}
  A \emph{rectangular Lefschetz decomposition} of $D^b(X)$ with respect to $\CO(1)$ is a semiorthogonal
  decomposition of the form
  \begin{equation*}
    D^b(X)=\langle \CA,\CA(1),\CA(2),\ldots,\CA(p-1)\rangle.
  \end{equation*}
  The subcategory $\CA(i-1)$ is called the \emph{$i$-th block} of the decomposition.
\end{definition}

In the best case scenario the first block of a rectangular Lefschetz decomposition 
is generated by an exceptional collection. Assume that $\CA=\langle E_0,E_1,\ldots,E_n\rangle$.
Then $D^b(X)$ admits a full exceptional collection
\begin{equation*}
  D^b(X)=
  \begin{pmatrix}
    E_n & E_n(1) & \cdots & E_n(p-1) \\
    \vdots & \vdots & \ddots & \vdots \\
    E_1 & E_1(1) & \cdots & E_1(p-1) \\
    E_0 & E_0(1) & \cdots & E_0(p-1)
  \end{pmatrix},
\end{equation*}
where the elements are ordered bottom to top, left to right.

Let us assume that the canonical line bundle $\omega_X\simeq \CO(-r)$ for some integer $r>0$,
which is called the \emph{index} of $X$.
Let $d$ denote the dimension of $X$. Then the maximal possible number of blocks in a Lefschetz
decomposition is $r$. Indeed, by Serre duality for any $F\in D^b(X)$ one has
\begin{equation*}
  \Ext^d(F(r), F)\simeq \Hom(F,F)^*\neq 0.
\end{equation*}

\begin{definition}
  Let $\langle E_0,E_1,\ldots,E_n\rangle=\CA\subseteq D^b(X)$ be an exceptional collection.
  We will say that it is a basis of a rectangular Lefschetz exceptional collection
  if there is a semiorthogonal decomposition
  \begin{equation*}
    \langle\CA, \CA(1),\ldots, \CA(r-1)\rangle\subseteq D^b(X),
  \end{equation*}
  where $r$ is the index of $X$.
\end{definition}

\begin{lemma}
  \label{lm:basis}
  A collection of objects $E_0, E_1, \ldots, E_n \in D^b(X)$ is a Lefschetz basis
  if and only if for all $0\leq i \leq j \leq n$ and all $0\leq t < r$
  \begin{equation*}
    \Ext^k(E_j(t), E_i) =
    \begin{cases}
      \kk, & \text{if}\ i = j,\ t = 0,\ \text{and}\ k = 0, \\
      0 & \text{otherwise.}
    \end{cases}
  \end{equation*}
\end{lemma}
\begin{proof}
  The remaining conditions $\Ext^\bullet(E_i(t), E_j)=0$ for $0\leq i < j\leq n$ and
  $0 < t < r$ follow immediately from Serre duality.
\end{proof}

\section{Odd isotropic Grassmannians}

\subsection{Definition}\label{ssec:def}

Let $V$ be a $(2n+1)$-dimensional vector space and let $\omega\in\Lambda^2V^*$ be
a skew-symmetric form of maximal rank. We denote by $K$ the one-dimensional kernel of $\omega$
and fix a non-zero vector $v\in K$. Given a pair $(V,\omega)$, one can develop the theory
of isotropic Grassmannians and, more generally, flag varieties, which is parallel
to the (even) symplectic case (see~\cite{mihai2007odd}).

We will use a couple of alternative descriptions of the variety $X=\IGr(k, V)$, parametrizing
$k$-dimensional subspaces isotropic with respect to $\omega$. Arguably, the most natural way
to define $X$ is to realize it as the zero set of a regular section of a vector bundle on a usual Grassmannian
variety. Let $\CU$ denote the tautological bundle of subspaces on the Grassmannian $\Gr(k, V)$.
Then $\omega\in \Lambda^2 V^* \simeq \Gamma(\Gr(k,V),\,\Lambda^2\CU^*)$ is a regular section and,
as in the even case, defines a closed embedding $\IGr(k, V)\hookrightarrow \Gr(k, V)$.
For instance, one has the universal sequence of vector bundles
\begin{equation*}
  0\to \CU\to V\otimes\CO_X\to \CQ\to 0,
\end{equation*}
were we denote the pullback of $\CU$ and $\CQ$ from $\Gr(k, V)$ by the same letters.

Alternatively, one can pick a $(2n+2)$-dimensional symplectic vector space $(V',\omega')$,
such that $V\subset V'$, and $\omega=\omega'|_V$. It is easy to check that $X$ embeds into
$\IGr(k, V')$ as the zero set of a regular section $s\in V'^*\simeq \Gamma(\IGr(k, V'),\, \CU^*)$,
where $s\in V^\perp\subset V'^*$ is any nonzero linear function.
Actually, one obtains a fiber square
\begin{equation*}
  \begin{tikzcd}[column sep=small]
     \IGr(k, V) \arrow[r, hook] \arrow[d, hook] & \IGr(k, V') \arrow[d, hook] \\
     \Gr(k, V) \arrow[r, hook] & \Gr(k, V')
  \end{tikzcd}
\end{equation*}
where both vertical arrows are induced by a regular section of $\Lambda^2\CU^*$, while
both horizontal arrows are induced by a regular section of $\CU^*$.

Let us denote by $j$ the embedding $X\hookrightarrow \IGr(k, V')$. We get a Koszul resolution
\begin{equation}\label{eq:koszul_schu}
   0 \to \Lambda^k \CU \to \cdots \to \Lambda^2 \CU \to \CU \to \CO_{\IGr(k, V')} \to j_*\CO_X \to 0.
\end{equation}

The odd isotropic Grassmannian $X$ carries a natural action of the non-reductive group of
transformations of $V$ preserving $\omega$ and is quasi-homogeneous with respect to this action.
There are only two orbits, and the closed one, denoted by $Z$, is defined by the condition that
the isotropic subspace contains $K$. It is a smooth subvariety of codimension $(2n-2k+2)$.

\subsection{A geometric construction}

Let $\bar{V}$ denote the quotient $V/K$, and let $\pi:V\to \bar{V}$ be the projection map.
The form $\omega$ descends to a symplectic form on $\bar{V}$.
Denote by $\Xb$ the even isotropic Grassmannian $\IGr(k, \bar{V})$. There is a natural incidence
subvariety in $\Xt\subset\IGr(k, V)\times\IGr(k, \bar{V})$, given by the condition
\begin{equation*}
  \Xt=\left\{ (U, U') \mid U\subset \pi^{-1}(U')\right\}.
\end{equation*}
We do not emphasize this point of view, but in the case $k=n$ the variety
$\Xt$ is actually the odd isotropic flag variety $\IFl(n, n+1; V)$.

Consider the diagram
\begin{equation*}
  \begin{tikzcd}[column sep=small]
    & \Xt \arrow[ld, "p"] \arrow[rd, "q"] & \\
    X & & \Xb
  \end{tikzcd}
\end{equation*}
It was shown in~\cite{mihai2007odd} that $\Xt$ on the one hand identifies with the relative
Grassmannian $\Gr_{\Xb}(k, \CUb\oplus\CO)$, so that $q$ is the natural projection, and, on the other
hand, the map $p$ is the blow-up of $Z$ in $X$.

Let us denote by $E$ the exceptional divisor of the blow-up $p$.
We get a larger diagram
\begin{equation}\label{eq:37diag}
   \begin{tikzcd}[column sep=small]
      & E \arrow[dl, "\pt"'] \arrow[rr, "\iot"] & & \Xt \arrow[dl, "p"'] \arrow[dr, "q"] & \\
      Z \arrow[rr, "\iot", hook] & & X & & \Xb
   \end{tikzcd}
\end{equation}

Denote by $\CUb$ the tautological subbundle on $\Xb$, as well as its pullback on $\Xt$.
Denote by $\CUt$ the inverse image of $\CUb$ under the projection of vector bundles
$V\otimes\CO_{\Xb}\to\bar{V}\otimes\CO_{\Xb}$, as well as its pullback on $\Xt$.
There is a short exact sequence
\begin{equation}\label{eq:ut}
   0\to \CO\to \CUt\to \CUb\to 0,
\end{equation}
which actually splits non-canonically (one has to choose a form-preserving
section of the map $\pi:V\to\bar{V}$).

Recall that $Z$ is the subvariety parametrizing the isotropic subspaces $U\subset V$ containing $K$.
Taking the quotient establishes an isomorphism $Z\simeq \IGr(k-1, \bar{V})$. In particular,
there is a sort exact sequence of vector bundles on $Z$
\begin{equation*}
  0\to \CO\to \CU\to \CW\to 0,
\end{equation*}
where $\CW$ is the tautological subbundle on $\IGr(k-1, \bar{V})$ (the sequence splits non-canonically).
In a similar fashion the exceptional divisor gets identified with $E\simeq \IFl(k-1, k; \bar{V})$,
and the universal isotropic flag is nothing but $\CW\subset\CUb$.

Let $\CO(H)$ and $\CO(\Hb)$ denote the very ample generators of $\Pic X$ and $\Pic \Xb$ respectively,
as well as their pullbacks on $\Xt$. Recall that $H\sim c_1(\CU^*)$, $\Hb\sim c_1(\CUb^*)$,
$K_X\sim -(2n-k+2)H$, and $K_\Xb\sim -(2n-k+2)\Hb$.
Computing the canonical class of $\Xt$ in two ways, one gets
\begin{equation*}
  E\sim H - \Hb.
\end{equation*}
Thus, there is a short exact sequence of vector bundles on $\Xt$
\begin{equation}\label{eq:uut}
   0 \to \CU\to \CUt \to \CO(E) \to 0,
\end{equation}
where the first map is the tautological embedding.

\subsection{Exact sequences}
In this section we construct several short exact sequences that will be useful later.
Let $X=\IGr(n,V)$, where $\dim V = 2n+1$. We denote by $\CQ$ the universal quotient bundle
on $X$, as well as its pullback on $\Xt$.

\begin{lemma}\label{lm:q}
   There is a short exact sequence of locally free sheaves on $\Xt$ of the form
   \begin{equation}
      \label{eq:q}
      0\to \CO(E) \to \CQ \to \CUb^* \to 0.
   \end{equation}
\end{lemma}

\begin{proof}
  Consider the commutative diagram of vector bundles
  \begin{equation*}
    \begin{tikzcd}
      0 \arrow[r] & \CU \arrow[r] \arrow[d, "\phi"] & V \arrow[d, "id"] \arrow[r] & \CQ \arrow[d, "\psi"] \arrow[r] & 0 \\
      0 \arrow[r] & \CUt \arrow[r] & V \arrow[r] & \CUb^* \arrow[r] & 0
    \end{tikzcd}
  \end{equation*}
  where the bottom row is induced from the sequence $0\to \CUb \to \bar{V}\to \CUb^* \to 0$.
  It follows form the snake lemma that $\psi$ is an epimorphism,
  and $\Ker \psi\simeq \Coker \phi$, while we know from~\eqref{eq:uut}
  that the latter is isomorphic to~$\CO(E)$.
\end{proof}

\begin{lemma}\label{lm:uub}
   There is a short exact sequence of coherent sheaves on $\Xt$ of the form
   \begin{equation}\label{eq:uub}
      0\to \CU\to \CUb \to \iot_* \CO(E) \to 0.
   \end{equation}
\end{lemma}
\begin{proof}
   Consider the morphism $\phi: \CU\to \CUb$ given by the composition of the
   inclusion $\CU\to \CUt$ with the projection $\alpha:\CUt \to \CUb$.
   Being a morphism of locally free sheaves which induces an isomorphism on fibers at
   the generic point, it is injective.
   Denote the cokernel of $\phi$ by $\CE$. Consider the diagram
   \begin{equation*}
      \begin{tikzcd}
         0 \rar & \CU \rar \dar{id} & \CUt \rar \dar{\alpha} & \CO(E)\rar \dar{\beta} & 0\\
         0 \rar & \CU \rar{\phi} & \CUb \rar & \CE \rar & 0
      \end{tikzcd}
   \end{equation*}
   It follows from the snake lemma that $\beta$ is surjective and
   $\Ker \beta \simeq \Ker \alpha \simeq \CO$. Thus, $\CE\simeq \iot_*\CO(E)$.
\end{proof}

The following lemma is trivial, however, we were not able to find it in the literature.
\begin{lemma}
  \label{lm:l2}
  Let $i:D\to X$ be an effective Cartier divisor on a scheme $X$. Let $\phi: \CE\to \CF$
  be an injective morphism of locally free sheaves of rank $r$ on $X$ such that
  the $\Coker\phi$ is isomorphic to $i_*\CL$ for a line bundle $\CL$ on $D$. Then
  $\Lambda^2\phi$ induces a short exact sequence
  \begin{equation*}
    0\to \Lambda^2\CE\to \Lambda^2\CF \to i_*(\mathcal{K}\otimes \CL)\to 0,
  \end{equation*}
  where $\mathcal{K}$ is the kernel of the morphism $i^*\CF\to \CL$.
  In particular, $\Coker \Lambda^2\phi$ is a rank $r-1$ vector bundle supported on $D$.
\end{lemma}
\begin{proof}
  Let us first check that $\Coker \Lambda^2\phi$ sheme-theoretically is supported on $D$. The question
  is local; thus, we can reformulate it in the following manner. Let $(A, m)$ be a local
  ring and let
  \begin{equation}
    \label{eq:l2local}
    \begin{tikzcd}
      0 \rar & M \rar{\phi} & M' \rar & A/(f) \rar & 0
    \end{tikzcd}
  \end{equation}
  be a short exact sequence of $A$ modules, where $f\in m$ is a non-zero-divisor, and
  $E$ and $F$ are free $A$-modules of rank $r$. Tensoring with $A/(f)$ over $A$, we get
  a right exact sequence of free $A/(f)$ modules
  \begin{equation*}
    \begin{tikzcd}
      M/(f)M \rar{\tilde{\phi}} &  M'/(f)M \rar & A/(f) \rar & 0.
    \end{tikzcd}
  \end{equation*}
  One can choose
  appropriate generators $\tilde{m}_1,\tilde{m}_2,\ldots, \tilde{m}_r$ and $\tilde{m}'_1,\tilde{m}'_2,\ldots, \tilde{m}'_r$ for
  $M/(f)M$ and $M'/(f)M'$ respectively so that $\tilde{\phi}(\tilde{m}_1)=0$ and
  $\tilde{\phi}(\tilde{m}_i)=\tilde{m}'_{i}$ for $i=2,\ldots, r$. By Nakayama's lemma, the generators
  $\tilde{m}_i$ and $\tilde{m}'_j$ lift to generators $m_i$ and $m'_j$ of $M$ and $M'$ respectively.
  In terms of these generators $\phi$ is given by a matrix of the form
  \begin{equation*}
    F =
    \begin{pmatrix}
      fa_{11} & fa_{12} & \cdots & fa_{1r} \\
      fa_{21} & 1+fa_{22} & \cdots & fa_{2r} \\
      \vdots & \vdots & \ddots & \vdots \\
      fa_{r1} & fa_{r2} & \cdots & 1+fa_{rr}
    \end{pmatrix}.
  \end{equation*}
  Remark that the diagonal elements $F_{ii}$ are invertible for $i=2,\ldots,r$.
  Using elementary row and column transformations, we can find generators
  $e_1,\ldots,e_r$ and $e'_1,\ldots,e'_r$ of $M$ and $M'$ respectively, such that
  $\phi(e_1) = fe'_1$ and $\phi(e_i)=e'_i$ for $i=2,\ldots,r$. Then
  $(\Lambda^2\phi)(e_1\wedge e_i) = fe'_1\wedge e'_i$ for $i=2,\ldots,r$, and
  $(\Lambda^2\phi)(e_i\wedge e_j) = e'_i\wedge e'_j$ for $2\leq i < j\leq r$.
  Thus, $\Coker \Lambda^2\phi$ is isomorphic to $\left(A/(f)\right)^{\oplus (r-1)}$.

  Now that we know that $\Coker \Lambda^2\phi$ is supported on $D$, we easily deduce
  that it is isomorphic to $i_*(\mathcal{K}\otimes\CL)$ from the short exact sequences
  \begin{equation*}
    0\to \mathcal{K} \to i^*\CF \to \CL \to 0\quad\text{and}\quad
    0\to \Lambda^2\mathcal{K}\to i^*\Lambda^2\CF \to \mathcal{K}\otimes\CL\to 0.
    \qedhere
  \end{equation*}
\end{proof}

Using the previous lemma, we can deduce the following.
\begin{lemma}
   There are short exact sequences of coherent sheaves on $\Xt$ of the form
   \begin{align}
      \label{eq:ubue}
      0 \to \CUb \to \CU(E) \to \iot_* \CW(E) \to 0 &, \\
      \label{eq:l2uub}
      0 \to \Lambda^2\CU \to \Lambda^2\CUb \to \iot_* \CW(E) \to 0 &.
   \end{align}
\end{lemma}
\begin{proof}
  Once we apply Lemma~\ref{lm:l2} to~\eqref{eq:uub}, we see that there is a commutative
  diagram of the form
  \begin{equation}
    \begin{tikzcd}
        & & 0 \dar & & \\
        & 0 \dar \rar & \CUb \dar \rar & \Ker\phi \dar & \\
        0 \rar & \Lambda^2\CU \rar \dar{id} & \Lambda^2\CUt \rar \dar & \CU(E) \rar \dar{\phi} & 0 \\
        0 \rar & \Lambda^2\CU \rar \dar & \Lambda^2\CUb \rar \dar & \iot_* W(E) \rar \dar & 0 \\
        & 0 \rar & 0 \rar & \Coker\phi &
    \end{tikzcd}
  \end{equation}
  It follows from the snake lemma that $\Ker\phi\simeq \CUb$, and $\Coker\phi = 0$.
\end{proof}

\section{Derived category of $\IGr(3, 7)$}\label{sec:37}

\subsection{Vanishing lemmas}

In this section we are going to prove several statements that hold
for a general submaximal odd isotropic Grassmannian $X=\IGr(n, 2n+1)$.
We begin with a vanishing result for even isotropic Grassmannians.

\begin{lemma}\label{lm:main_even}
   Let $(\lambda_1,\lambda_2,\ldots,\lambda_n)\in P^+_n$ be a dominant weight
   of the group $\GL_n$ such that
   \begin{enumerate}[label=\emph{(\roman*)}, ref=(\roman*)]
      \item \label{lm:main_even:c1} $\lambda_n < 0$,
      \item \label{lm:main_even:c2} $\lambda_1 \geq -(n+2)$,
      \item \label{lm:main_even:c3} $\lambda_i-\lambda_{i+1} \leq 3$ for $i=1,\ldots,n-1$.
   \end{enumerate}
   Then the bundle $\Sigma^\lambda\CU^*$ on $\IGr(n, 2n+2)$ is acyclic.
\end{lemma}
\begin{proof}
   Assume that the bundle $\Sigma^\lambda \CU^*$ is not acyclic.
   Then, according to Theorem~\ref{thm:bbwc}, the absolute values of the elements in the sequence
   \begin{equation}\label{eq:lm_even:bbw}
       \gamma = (n+1 + \lambda_1, n + \lambda_2,\ldots, 2 + \lambda_n, 1)
   \end{equation}
   are strictly positive and distinct.
   In particular,
   \begin{equation}\label{eq:lm_even:abs}
      \gamma_i \geq 2 \quad \text{or} \quad \gamma_i \leq -2
   \end{equation}
   for all $i=1,\ldots,n$.
   It follows from condition~\ref{lm:main_even:c1} that $\gamma_n=2+\lambda_n \leq 1$,
   while condition~\ref{lm:main_even:c2} implies that $\gamma_1=n+1+\lambda_1\geq -1$.
   Combining the latter inequalities with~\eqref{eq:lm_even:abs}, we conclude that
   \begin{equation}
      \gamma_1 \geq 2 \quad \text{and} \quad \gamma_n \leq -2. 
   \end{equation}
   
   As the first $n$ elements in the sequence~\eqref{eq:lm_even:bbw} are strictly
   decreasing, there exists an index
   $1\leq j\leq n-1$ such that
   \begin{equation}\label{eq:lm_even:gap}
      \gamma_j = n+2-j+\lambda_j \geq 2 \quad \text{and} \quad \gamma_{j+1} = n+2-(j+1)+\lambda_{j+1}\leq -2.
   \end{equation}
   Thus, $\lambda_j-\lambda_{j+1}\geq 3$. From condition~\ref{lm:main_even:c3} we
   conclude that $\lambda_j-\lambda_{j+1}=3$, which implies that the inequalities~\eqref{eq:lm_even:gap}
   are, in fact, equalities. Finally, in such case $\abs{\gamma_j}=\abs{\gamma_{j+1}}$, which contradicts
   the assumption that all the absolute values $\abs{\gamma_i}$ for $i=1,\ldots,n+1$ are distinct.
\end{proof}

From the previous lemma we deduce a vanishing result for odd isotropic Grassmannians.

\begin{lemma}\label{lm:main}
   Let $(\lambda_1,\lambda_2,\ldots,\lambda_n)\in P^+_n$ be a dominant weight
   of the group $\GL_n$ such that
   \begin{enumerate}[label=\emph{(\roman*)}, ref=(\roman*)]
      \item \label{lm:main:c1} $\lambda_n < 0$,
      \item \label{lm:main:c2} $\lambda_1 \geq -(n+1)$,
      \item \label{lm:main:c3} $\lambda_i-\lambda_{i+1} \leq 2$ for $i=1,\ldots,n-1$.
   \end{enumerate}
   Then the bundle $\Sigma^\lambda\CU^*$ on $X=\IGr(n, 2n+1)$ is acyclic.
\end{lemma}
\begin{proof}
   Choose an embedding $j : X \hookrightarrow Y=\IGr(2, 2n+2)$
   as the zero scheme of a regular section $s\in \Gamma(Y,\,\CU^*)$
   (see Section~\ref{ssec:def}). Then
   \begin{equation*}
      H^\bullet(X,\, \Sigma^\lambda \CU^*) \simeq H^\bullet(X,\, j^*\Sigma^\lambda \CU^*) \simeq
      H^\bullet(Y,\, j_*j^*\Sigma^\lambda \CU^*) \simeq H^\bullet(Y,\, \Sigma^\lambda \CU^*\otimes j_*\CO_X).
   \end{equation*}
   Replacing $j_*\CO_X$ by its Koszul resolution~\eqref{eq:koszul_schu} and looking at the corresponding
   spectral sequence, we see that it is enough to show that the bundles
   $\Sigma^\lambda \CU^*\otimes \Lambda^p \CU$ on $Y$ are acyclic for all $p=0,\ldots,n$.

   Let $\Sigma^\alpha \CU^* \subset \Sigma^\lambda \CU^* \otimes \Lambda^p \CU$ be an irreducible
   summand.
   It follows from Lemma~\ref{lm:pieri} that
   \begin{equation}\label{eq:lm_main:summand}
      \lambda_i-1 \leq \alpha_i \leq \lambda_i
   \end{equation}
   for all $i=1,\ldots,n$. Combining~\eqref{eq:lm_main:summand} with conditions \ref{lm:main:c1},
   \ref{lm:main:c2}, and \ref{lm:main:c3}, we conclude that $\alpha$ satisfies
   the conditions listed in Lemma~\ref{lm:main}; thus follows the required vanishing.
\end{proof}

\begin{proposition}\label{prop:basis_right}
   The bundles $\langle \CO,\, \CU^*,\, \Lambda^2 \CU^*,\, \ldots,\, \Lambda^{n-1} \CU^* \rangle$ 
   form a basis of a rectangular Lefschetz exceptional collection in $D^b(X)$. 
\end{proposition}
\begin{proof}
  Recall that $\omega_X\simeq \CO(-n-2)$. Following Lemma~\ref{lm:basis},
  we need to show that for $0\leq p\leq q \leq n-1$ and $0\leq t < n+2$ one has
  \begin{equation}\label{eq:basis_right}
    \Ext^\bullet(\Lambda^q\CU^*(t), \Lambda^p\CU^*) =
    \begin{cases}
      \kk, & \text{when } p=q \text{ and } t = 0, \\
      0, & \text{otherwise}.
    \end{cases}
  \end{equation}
  As $\Ext^\bullet(\Lambda^q\CU^*(t), \Lambda^p\CU^*)\simeq H^\bullet(X, \Lambda^p\CU^*\otimes\Lambda^q\CU(-t))$,
  it is sufficient to compute cohomology of every irreducible summand
  $\Sigma^\alpha\CU^*\subseteq \Lambda^p\CU^*\otimes\Lambda^q\CU(-t)$.
  It follows from Pieri's formulas that, unless $t=0$, $p=q$, and $\alpha=(0,\ldots,0)$,
  the weight $\alpha$ satisfies the conditions of Lemma~\ref{lm:main}.
  Thus, \eqref{eq:basis_right} holds.
\end{proof}

From now on we assume additionally that $n\geq 3$.
We omit the proofs of the following two propositions, as they are very similar to the proof
of the previous one.

\begin{proposition}\label{prop:basis}
  Let $p$ and $q$ be nonnegative integers such that $p+q=n-1$. Then the bundles
  \begin{equation*}
     \langle \Lambda^p \CU,\, \Lambda^{p-1} \CU,\, \ldots,\, \CU,\, \CO,\,
        \CU^*,\, \ldots,\, \Lambda^{q-1} \CU^*,\, \Lambda^{q} \CU^* \rangle
  \end{equation*}
  form a basis of a rectangular Lefschetz exceptional collection in $D^b(X)$.
\end{proposition}

\begin{proposition}\label{prop:s2}
  The bundles $\langle S^2\CU,\, \CU,\, \CO \rangle$
  form a basis of a rectangular Lefschetz exceptional collection in $D^b(X)$.
\end{proposition}

\begin{lemma}\label{lm:s2mut}
  Consider the subcategory $\CA = \langle \CU,\, \CO \rangle \subseteq D^b(X)$.
  Then $R_{\CA}(S^2\CU) \simeq \Lambda^2\CQ$.
\end{lemma}
\begin{proof}
  Restricting the exact sequence 
  \begin{equation*}
    0\to S^2\CU\to V\otimes\CU\to \Lambda^2V\otimes\CO\to \Lambda^2\CQ\to 0
  \end{equation*}
  from $\Gr(n, 2n+1)$ to $X$, we see that it is enough to check that
  $\Lambda^2\CQ\in \lort{\langle\CU, \CO\rangle}$.
  Consider the embedding $j:X\to \Gr(n,2n+1)$. Then
  \begin{equation*}
    \Ext^\bullet(\Lambda^2\CQ, \CO)\simeq H^\bullet(X, \Lambda^2\CQ^*)\simeq
    H^\bullet(\Gr(n, 2n+1), \Lambda^2\CQ^*\otimes j_*\CO).
  \end{equation*}
  Similarly,
  \begin{equation*}
    \Ext^\bullet(\Lambda^2\CQ, \CU)\simeq H^\bullet(\Gr(n, 2n+1),\, \CU\otimes\Lambda^2\CQ^*\otimes j_*\CO).
  \end{equation*}
  Replacing $j_*\CO$ with its Koszul resolution, we reduce the problem to showing
  that for all irreducible summands $\Sigma^\alpha\CU\in \Lambda^p\Lambda^2\CU$,
  where $s=0,\ldots,\rk\Lambda^2\CU$,
  \begin{equation*}
    H^\bullet(\Gr(n, 2n+1), \Lambda^2\CQ^*\otimes\Sigma^\alpha\CU) = 
    H^\bullet(\Gr(n, 2n+1), \Lambda^2\CQ^*\otimes\Sigma^\alpha\CU\otimes\CU) = 0. 
  \end{equation*}
  Denote $\mu=(1,1,0,\ldots,0)$ and let $\lambda$ be either $\alpha$ or the weight
  of an irreducible summand of $\Sigma^\alpha\CU\otimes\CU$.
  Again, it is enough to show that
  \begin{equation}\label{eq:s2mut}
    H^\bullet(\Gr(n,2n+1), \Sigma^\lambda\CU\otimes\Sigma^\mu\CQ^*) = 0.
  \end{equation}

  It follows from Lemmas~\ref{lm:ll2} and~\ref{lm:pieri} that $\lambda$ and $\mu$
  satisfy the conditions of Lemma~\ref{lm:kap-dual}. Thus, \eqref{eq:s2mut}
  holds unless $\lambda=\mu^T$. The latter implies that
  \(
    2=|\mu^T|=|\lambda| \in \{2s,2s+1\}
  \). It remains consider the case $|\lambda|=2$, in which
  $\lambda=(1,1,0,\ldots,0)\neq (2, 0,\ldots,0)=\mu^T$.
\end{proof}

\begin{corollary}\label{cor:l2}
  The bundles $\langle \CU,\, \CO,\, \Lambda^2\CQ \rangle$ form a basis
  of a rectangular Lefschetz exceptional collection in $D^b(X)$.
\end{corollary}
\begin{proof}
  Follows from Proposition~\ref{prop:s2}, as the two Lefschetz
  collections are related by blockwise mutations.
\end{proof}

Finally, we can prove the following.

\begin{proposition}\label{prop:37eb}
  The bundles $\langle \CU,\, \CO,\, \CU^*,\, \Lambda^2\CQ \rangle$ form
  a basis of a rectangular Lefschetz exceptional collection in $D^b(X)$.
\end{proposition}
\begin{proof}
  It follows from a combination of Lemmas~\ref{lm:basis}, \ref{prop:basis}, and Corollary~\ref{cor:l2}
  that it suffices to check the vanishing
  \begin{equation*}
    \Ext^\bullet(\Lambda^2\CQ(t), \CU^*)\simeq H^\bullet(X, \Lambda^2\CQ^*(-t)\otimes\CU^*)\simeq
    H^\bullet(\Gr(n, 2n+1),\, \Lambda^2\CQ^*(-t)\otimes\CU^*\otimes j_*\CO)=0
  \end{equation*}
  for $t=0,\ldots, n+1$, where $j:X\to\Gr(n, 2n+1)$ is the usual embedding.
  As in the proof of Lemma~\ref{lm:s2mut}, we replace $j_*\CO$ with its
  Koszul resolution and desire to show that 
  \begin{equation*}
    H^\bullet(\Gr(n, 2n+1),\, \Lambda^2\CQ^*(-t)\otimes\CU^*\otimes\Sigma^\alpha\CU)=0
  \end{equation*}
  for any irreducible summand $\Sigma^\alpha\CU\subseteq \Lambda^s\Lambda^2\CU$,
  $s=0,\ldots,\rk \Lambda^2\CU$.

  First assume that $t > 0$. Then
  $\Lambda^2\CQ^*(-t)\otimes\CU^*\otimes\Sigma^\alpha\CU\simeq
  \Lambda^2\CQ^*(-t+1)\otimes\Lambda^{n-1}\CU\otimes\Sigma^\alpha\CU$. Let
  $\Sigma^\lambda\CU$ be an irreducible summand in $\Lambda^{n-1}\CU\otimes\Sigma^\alpha\CU$
  and let $\mu=(t,t,t-1,\ldots,t-1)$. Remark that $\Sigma^\mu\CQ^*\simeq \Lambda^2\CQ^*(-t+1)$.
  It follows from Pieri's rule and Lemma~\ref{lm:ll2} that $\lambda_1\leq n$.
  If $t>1$, then $(\mu^T)_1=n+1$, which implies $\lambda\neq\mu^T$. Thus,
  it follows from Lemma~\ref{lm:kap-dual} that the bundle
  $\Sigma^\mu\CQ^*\otimes\Sigma^\lambda\CU$ is acyclic.
  If $t=1$, then $|\lambda|=2s+n-1\geq 2$, while $|\mu^T|=2$. Thus,
  if $\lambda=\mu^T$, then $s=0$ and $n=3$, which implies that
  $\lambda=(1,1,0,\ldots,0)\neq (2,0,\ldots,0)=\mu^T$.

  It remains to deal with the case $t=0$. Let $\mu=(1,1,0,\ldots,0)$, and let
  $\Sigma^\lambda\CU$ be an irreducible summand in $\CU^*\otimes\Sigma^\alpha\CU$.
  It follows from Pieri's formulas and Lemma~\ref{lm:ll2} that
  $n-1\geq \lambda_1$ and $\lambda_{n-1}\geq 0$, while
  $(\lambda_1,\ldots,\lambda_{n-1})\neq (\mu_1,\ldots,\mu_{n-1})^T$.
  To get the desired vanishing apply Lemma~\ref{lm:kap-gen} for $p=q=n-1$. 
\end{proof}

\subsection{Fullness}\label{ssec:37full}
The main result of this section is the following statement.  

\begin{proposition}\label{prop:37full}
  The rectangular Lefschetz exceptional collection with the basis
  $\langle \CU,\, \CO,\, \CU^*,\, \Lambda^2\CQ \rangle$ is full in $D^b(\IGr(3, 7))$.
  Namely,
  \begin{equation}\label{eq:ec}
    D^b(\IGr(3, 7)) =
      \begin{pmatrix*}[r]
        \Lambda^2\CQ & \Lambda^2\CQ(1) & \Lambda^2\CQ(2) & \Lambda^2\CQ(3) & \Lambda^2\CQ(4) \\
        \CU^* & \CU^*(1) & \CU^*(2) & \CU^*(3) & \CU^*(4) \\
        \CO & \CO(1) & \CO(2) & \CO(3) & \CO(4) \\
        \CU & \CU(1) & \CU(2) & \CU(3) & \CU(4) \\
      \end{pmatrix*},
  \end{equation}
  where the exceptional objects are ordered bottom to top, left to right.
\end{proposition}

The proof of Proposition~\ref{prop:37full} that we present here is rather straightforward.
It can be shortened a bit using so called \emph{staircase complexes} and their relative
versions (see~\cite{fonarevKP}). Unfortunately, we do not have a general enough construction yet,
so we choose to take the simpler path and avoid extra definitions.

\begin{proof}
   Let $\CA\subseteq D^b(X)$ denote the full triangulated subcategory generated by
   the Lefschetz exceptional collection~\eqref{eq:ec}. On the one hand,
   from Orlov's blow-up formula we know that there is a semiorthogonal decomposition
   \begin{equation*}
      D^b(\Xt) = \langle \iot_*\pt^* D^b(Z)(E),\, p^*D^b(X) \rangle.
   \end{equation*}
   As the functor $p^*:D^b(X)\to D^b(\Xt)$ is fully faithful, in order to prove
   that $\CA=D^b(X)$ it suffices to show that the category
   \begin{equation*}
      \CAt = \langle \iot_*\pt^* D^b(Z)(E),\, p^*\CA \rangle.
   \end{equation*}
   coincides with $D^b(\Xt)$.
   
   On the other hand, $\Xt = \PP_\Xb(\CUt^*)$, and the corresponding Grothendieck
   line bundle is isomorphic to $\CO(E)$. Orlov's projective bundle formula provides
   the semiorthogonal decomposition
   \begin{equation}\label{eq:37p}
      D^b(\Xt) = \langle D^b(\Xb),\, D^b(\Xb)(E),\, D^b(\Xb)(2E),\, D^b(\Xb)(3E) \rangle.
   \end{equation}
   Recall that $\Xb$ is nothing but the even Lagrangian Grassmannian $\LGr(3, 6)$.
   It was shown in~\cite{samokhin2001derived} that its derived category admits
   a full Lefschetz rectangular exceptional collection
   \begin{equation}\label{eq:36}
      D^b(\Xb) = 
      \begin{pmatrix*}[r]
         \CO & \CO(\Hb) & \CO(2\Hb) & \CO(3\Hb) \\
         \CUb & \CUb(\Hb) & \CUb(2\Hb) & \CUb(3\Hb) \\
      \end{pmatrix*}.
    \end{equation}

   Let $\CB_i$ denote the $i$-th component $D^b(\Xb)(iE)$ of the decomposition~\eqref{eq:37p}.
   Here is our choice of full exceptional collections for $\CB_i$: for the first 3 blocks
   we take the collection~\eqref{eq:36}, while for $\CB_3$ we take the collection~\eqref{eq:36}
   further twisted by $\CO(\Hb)$. Finally, we mutate the objects $\CO(2\Hb)$, $\CUb(2\Hb)$,
   $\CO(3\Hb)$, and $\CUb(3\Hb)$ through $\langle\CB_1, \CB_2, \CB_3\rangle$ using
   Lemma~\ref{lm:mutproj} to obtain the full exceptional collection
   \begin{equation}
     \label{eq:xtec}
     D^b(\Xt) =
     \begin{pmatrix*}[l]
      & & & \CO(3\Hb+4E) & \CO(4\Hb+4E) \\
      & & & \CUb(3\Hb+4E) & \CUb(4\Hb+4E) \\[4pt]
      & \CO(\Hb+3E) & \CO(2\Hb+3E) & \CO(3\Hb+3E) & \CO(4\Hb+3E) \\
      & \CUb(\Hb+3E) & \CUb(2\Hb+3E) & \CUb(3\Hb+3E) & \CUb(4\Hb+3E) \\[4pt]
      \CO(2E) & \CO(\Hb+2E) & \CO(2\Hb+2E) & \CO(3\Hb+2E) \\
      \CUb(2E) & \CUb(\Hb+2E) & \CUb(2\Hb+2E) & \CUb(3\Hb+2E) \\[4pt]
      \CO(E) & \CO(\Hb+E) & \CO(2\Hb+E) & \CO(3\Hb+E) \\
      \CUb(E) & \CUb(\Hb+E) & \CUb(2\Hb+E) & \CUb(3\Hb+E) \\[4pt]
      \CO & \CO(\Hb) & & \\
      \CUb & \CUb(\Hb) & & \\
     \end{pmatrix*}
   \end{equation}
   
   We are going to show that $\CAt=D^b(\Xt)$ by showing that every object in
   the collection~\eqref{eq:xtec} belongs to $\CAt$. Recall that there is a short exact
   sequence of vector bundles
   \begin{equation}
     \label{eq:ubub*}
     0 \to \CUb \to \Vb\otimes\CO \to \CUb^* \to 0
   \end{equation}
   on $\Xt$, which is just the pullback of the corresponding sequence from $\Xb$.
   In the following we repeatedly use the following trivial fact, which immediately follows
   from the existence of the short exact sequence~\eqref{eq:ubub*}: for any line bundle
   $\CL \in \Pic \Xt$ if two of the bundles $\CUb\otimes \CL$, $\CL$, $\CUb^*\otimes \CL$
   belong to $\CAt$, then the third one belongs to $\CAt$ as well. Finally, it will be convenient
   to us to rewrite the collection~\eqref{eq:xtec} using the rational equivalence
   $\Hb\sim H - E$.

   \begin{equation}
     \label{eq:xtech}
     D^b(\Xt) =
     \begin{pmatrix*}[l]
      & & & \CO(3H+E) & \CO(4H) \\
      & & & \CUb(3H+E) & \CUb(4H) \\[4pt]
      & \CO(H+2E) & \CO(2H+E) & \CO(3H) & \CO(4H-E) \\
      & \CUb(H+2E) & \CUb(2H+E) & \CUb(3H) & \CUb(4H-E) \\[4pt]
      \CO(2E) & \CO(H+E) & \CO(2H) & \CO(3H-E) \\
      \CUb(2E) & \CUb(H+E) & \CUb(2H) & \CUb(3H-E) \\[4pt]
      \CO(E) & \CO(H) & \CO(2H-E) & \CO(3H-2E) \\
      \CUb(E) & \CUb(H) & \CUb(2H-E) & \CUb(3H-2E) \\[4pt]
      \CO & \CO(H-E) & & \\
      \CUb & \CUb(H-E) & & \\
     \end{pmatrix*}
   \end{equation}

   The rest of the proof consists of showing that the bundles from the collection~\eqref{eq:xtech} belong to $\CAt$.
   In each case we produce a short exact sequece (or a complex) with all terms but
   one known to be in $\CAt$ and conclude that the remaining term in $\CAt$ as well.
   To simplify notation, we denote the subcategory $\iot_*\pt^*D(Z)(E)$ by $\CB$.

   \emph{Step 1.} The bundles $\boxed{\CO$, $\CO(H)$, $\CO(2H)$, $\CO(3H)$, $\CO(4H)}$ belong to $\CAt$,
   as they belong to $p^*\CA$.

   \emph{Step 2.} The bundles $\boxed{\CUb$, $\CUb(H)$, $\CUb(2H)$, $\CUb(3H)$, $\CUb(4H)}$
   belong to $\CAt$.
   
   Twisting the short exact sequence~\eqref{eq:uub} by $\CO(tH)$ for $t=0,\ldots,4$,
   we get the sequence
   \begin{equation}
     \label{eq:st2}
     0\to \CU(tH) \to \CUb(tH) \to \iot_*\CO(tH + E) \to 0.      
   \end{equation}
   The first sheaf in the sequence~\eqref{eq:st2} belongs to $p^*\CA$, while the third sheaf
   is in $\CB$.
   
   \emph{Step 3.} The bundles $\boxed{\CO(E)$, $\CO(H+E)$, $\CO(2H+E)$, $\CO(3H+E)$, $\CO(4H+E)}$
   belong to $\CAt$.\nopagebreak
   
   It suffices to twist the short exact sequence
   \begin{equation*}
      0\to \CO\to \CO(E)\to \iot_*\CO(E)\to 0
   \end{equation*}
   by $\CO(tH)$ for $t=0,\ldots,4$, and use Step~1.

   \emph{Step 4.} The bundles
   $\boxed{\CUb(E)$, $\CUb(H+E)$, $\CUb(2H+E)$, $\CUb(3H+E)$, $\CUb(4H+E)}$ belong to $\CAt$.
   
   Dualizing the short exact sequence~\eqref{eq:ubue} and using the Grothendieck
   duality, we obtain a short exact sequence of the form
   \begin{equation}\label{eq:ueub}
      0\to \CU^*(-E)\to \CUb^*\to \iot_*\CW^*\to 0.
   \end{equation}
   Twisting the sequence~\eqref{eq:ueub} by $\CO(E+tH)$ for $t=0,\ldots,4$, we see that the bundles
   $\CUb^*(E)$, $\CUb^*(H+E)$, $\CUb^*(2H+E)$, $\CUb^*(3H+E)$, and $\CUb^*(4H+E)$ belong to $\CAt$.
   From Step~3 we conclude that the bundles
   $\CUb(E)$, $\CUb(H+E)$, $\CUb(2H+E)$, $\CUb(3H+E)$, and $\CUb(4H+E)$ belong to $\CAt$.

   \emph{Step 5.} The bundles $\boxed{\CO(2E)$ and $\CO(H+2E)}$ belong to $\CAt$.
   
   Dualizing the short exact sequence~\eqref{eq:uub} and using the Grothendieck
   duality, we obtain a short exact sequence of the form
   \begin{equation}\label{eq:ubu}
      0\to \CUb^*\to \CU^*\to \iot_*\CO\to 0.
   \end{equation}
   Twisting by $\CO(tH)$ and using Step~2 we conclude that $\iot_*\CO(tH)\in\CAt$ for $t=0,\ldots,4$.
   From Lemma~\ref{lm:blmut}
   we see that $\iot_*\CO(tH+2E)\in\CAt$ for $t=0,\ldots,4$.
   Finally, from Step~3 and the short exact sequence
   \begin{equation*}
      0\to \CO(tH+E)\to \CO(tH+2E)\to \iot_*\CO(tH+2E)\to 0
   \end{equation*}
   we conclude that $\CO(2E)$ and $\CO(H+2E)$ belong to $\CAt$.

   \emph{Step 6.} The bundles $\boxed{\CO(H-E)$, $\CO(2H-E)$, $\CO(3H-E)$, $\CO(4H-E)}$ belong to $\CAt$.
   
   From Step~5 we know that $\iot_*\CO(tH)\in\CAt$ for $t=0,\ldots,4$.
   The claim follows from the short exact sequence
   \begin{equation*}
      0\to \CO(tH-E)\to \CO(tH)\to \iot_*\CO(tH)\to 0
   \end{equation*}
  and Step~1.

   \emph{Step 7.} The bundles $\boxed{\CUb(H-E)$, $\CUb(2H-E)$, $\CUb(3H-E)$, $\CUb(4H-E)}$ belong to $\CAt$.
   
   Applying $\Lambda^2$ to the short exact sequence~\eqref{eq:q}, we get
   \begin{equation}\label{eq:q2}
      0\to \CUb^*(E)\to \Lambda^2\CQ\to \Lambda^2\CUb^*\to 0.
   \end{equation}
   Twisting by $\CO(tH)$ and using Step~4, we see that $\Lambda^2\CUb^*(tH)\in\CAt$ for $t=0,\ldots,4$.
   Now just use the isomorphism $\Lambda^2\CUb^*\simeq \CUb(\Hb)=\CUb(H-E)$.

   \emph{Step 8.} The bundles $\boxed{\CU(2E)$ and $\CU(H+2E)}$ belong to $\CAt$.
   
   Applying $\Lambda^2$ to the dual of the short exact sequence~\eqref{eq:uut},
   and using the isomorphism $\Lambda^2\CU^*\simeq \CU(H)$ we get
   \begin{equation}\label{eq:utu2}
      0\to \CU^*(-E)\to \Lambda^2\CUt^*\to \CU(H)\to 0.
   \end{equation}
   From the non-canonical splitting $\CUt^*\simeq\CUb^*\oplus\CO$
   we have $\Lambda^2\CUt^*\simeq \Lambda^2\CUb^*\oplus\CUb^*$.
   Twisting~\eqref{eq:utu2} by $\CO(tH)$ and using Steps~2 and~7, we conclude that
   $\CU^*(tH-E)\in\CAt$ for $t=0,\ldots,3$.
   Thus, twisting~\eqref{eq:ueub} by $\CO(tH)$ and using Steps~1 and~2,
   we see that $\iot_*\CW^*(tH)\in\CAt$ for $t=0,\ldots,3$. It follows from
   Lemma~\ref{lm:blmut} that $\iot_*\CW^*(tH+2E)\in\CAt$ for $t=0,\ldots,3$.

   Finally, consider the short exact sequence
   \begin{equation*}
      0\to \CUb^*(tH+E)\to \CUb^*(tH+2E)\to \iot_*\CUb^*(tH+2E)\to 0.
   \end{equation*}
   From Step~4 and the short exact sequence
   \begin{equation*}
      0\to \iot_*\CO(tH+E)\to \iot_*\CUb^*(tH+2E)\to \iot_*\CW^*(tH+2E)\to 0
   \end{equation*}
   we conclude that $\CUb^*(2E)$ and $\CUb^*(H+2E)$ belong to $\CAt$. It follows from
   Step~5 that the bundles $\CU(2E)$ and $\CU(H+2E)$ belong to $\CAt$.

   \emph{Step 9.} The bundle $\boxed{\CO(3H-2E)}$ belongs to $\CAt$
   
   From Step~8 we know that $\CU^*(3H-E)\in\CAt$. From Steps~6 and 7 we know that
   $\CUt^*(3H-E)\in\CAt$. Thus, twisting the dual of~\eqref{eq:uut} by $\CO(3H-E)$, we get
   \begin{equation*}
     0\to \CO(3H-2E) \to \CUt^*(3H-E)\to \CU^*(3H-E)\to 0, 
   \end{equation*}
   and conclude that $\CO(3H-2E)\in\CAt$.

   \emph{Step 10.} The bundle $\boxed{\CUb(3H-2E)}$ belongs to $\CAt$.
   
   We show that $\CUb(3H-2E)\simeq\CUb(3\Hb+E)$ is in $\CAt$.
   Consider the Koszul complex associated to the nowhere vanishing section $\CO\to \CUt^*(E)$,
   further twisted by $\CO(\Hb+4E)$:
   \begin{equation}
     \label{eq:st10}
     0\to \CO\to \Lambda^3\CUt(\Hb+E)\to \Lambda^2\CUt(\Hb+2E)\to \CUt(\Hb+3E)\to \CO(\Hb+4E)
     \to 0.
   \end{equation}
   Tensoring~\eqref{eq:st10} with $\CUb(3\Hb+E)$, we obtain the complex
   \begin{equation}
     \label{eq:st10a}
     0\to \CUb(3\Hb+E)\to \CUb\otimes\Lambda^3\CUt(4\Hb+2E)\to \CUb\otimes\Lambda^2\CUt(4\Hb+3E) \to
     \CUb\otimes\CUt(4\Hb+4E)\to \CUb(4\Hb+5E) \to 0.
   \end{equation}
   Remark that the bundles $\CUb\otimes\Lambda^3\CUt(4\Hb+2E)$ and $\CUb\otimes\Lambda^2\CUt(4\Hb+3E)$
   belong to the subcategories $\CB_2$ and $\CB_3$ respectively, and from the previous steps
   we know that $\CB_2$ and $\CB_3$ are contained in $\CAt$. Moreover, from Step~4 we know that
   the bundle $\CUb(4\Hb+5E)\simeq \CUb(4H+E)$ is in $\CAt$ as well.
   Thus, if we manage to show that $\CUb\otimes\CUt(4\Hb+4E)$ is in $\CAt$, we will be able
   to deduce, using the sequence~\eqref{eq:st10a}, that $\CUb(3H-2E)$ is in $\CAt$, which
   will finish the proof.
   Choose a splitting $\CUt\simeq\CUb\oplus\CO$. Then
   \begin{equation}
     \label{eq:st10b}
     \CUb\otimes\CUt(4\Hb+4E)\simeq \CUb(4\Hb+4E)\oplus \Lambda^2\CUb(4\Hb+4E)\oplus
     S^2\CUb(4\Hb+4E).
   \end{equation}
   We will show that all the summands in~\eqref{eq:st10b} belong to $\CAt$.
   From Step~2 we know that $\CUb(4\Hb+4E)\simeq \CUb(4H)$ belongs to $\CAt$. From the
   isomorphism $\Lambda^2\CUb(4\Hb+4E)\simeq \CUb^*(3\Hb+4E)\simeq \CUb^*(3H+E)$ and
   Step~4 we see that the second summand in~\eqref{eq:st10b} belongs to $\CAt$ as well.

   Finally, from the short exact sequence
   \begin{equation*}
     0\to \CUb\to \Vb\otimes\CO \to \CUb^*\to 0
   \end{equation*}
   we get the exact sequence
   \begin{equation}
     \label{eq:st10c}
     0\to S^2\CUb\to \Vb\otimes\CUb\to \Lambda^2\Vb\otimes\CO\to \Lambda^2\CUb^*\to 0.
   \end{equation}
   Twisting~\eqref{eq:st10c} by $\CO(4H)$, we get the sequence
   \begin{equation}
     \label{eq:st10d}
     0\to S^2\CUb(4H)\to \Vb\otimes\CUb(4H)\to \Vb\otimes\CO(4H)\to \Lambda^2\CUb^*(4H)\to 0.
   \end{equation}
   From Steps~1 and 4 we know that the second and third terms of the sequence~\eqref{eq:st10d}
   belong to $\CAt$. From Step~7 we know that $\Lambda^2\CUb^*(4H)$ is in $\CAt$ as well.
   Thus, $S^2\CUb(4H)\simeq S^2\CUb(4\Hb+4E)$ belongs to $\CAt$, which finishes the proof.
\end{proof}

Looking at the index and rank of the Grothendieck group, the following seems plausible.

\begin{conjecture}
  The bounded derived category of the odd isotropic Grassmannian $\IGr(n, 2n+1)$
  admits a full rectangular Lefschetz exceptional collection.
\end{conjecture}

\bibliographystyle{abbrv}
\bibliography{igr37}

\end{document}